\documentclass[12pt,oneside]{article}
\usepackage{amsmath,amssymb,amsfonts,amsthm}
\newcommand{\T}{{\cal T}}

\newcommand{\set}[1]{\left\{#1\right\}}

\newcommand {\cp}{\mathfrak{X}(\pi (M))}

\def\o#1{\overline{#1}}

\newtheorem{thm}{Theorem}[section]

\newtheorem{lem}[thm]{Lemma}
\newtheorem{prop}[thm]{Proposition}
\newtheorem{defn}[thm]{Definition}

\numberwithin{equation}{section}

\begin{document}
\title{{{\bf An intrinsic proof of  Numata's theorem on Landsberg spaces}}}
\author{\bf{ A. Soleiman$^{1,3}$ and S. G. Elgendi$^{2,3}$ }}
\date{}
\maketitle                     
\vspace{-1.0cm}
\begin{center}
{$^{1}$ Department of Mathematics, College of Science and Arts - Qurayyat,
Al Jouf University, Skaka,  Kingdom of Saudia Arabia
\vspace{0.2cm}
\\ $^{2}$ Department of Mathematics, Faculty of Science, Islamic University of Madinah,  Madinah, Kingdom of Saudia Arabia
\vspace{0.2cm}
\\ $^{3}$ Department of Mathematics, Faculty of Science, Benha University, Benha,  Egypt}
\end{center}
\vspace{-0.8cm}
\begin{center}
E-mails: amr.hassan@fsc.bu.edu.eg, amrsoleiman@yahoo.com\\
{\hspace{1.8cm}}salah.ali@fsc.bu.edu.eg, salahelgendi@yahoo.com
\end{center}

\vspace{0.7cm} \maketitle
\smallskip

\noindent{\bf Abstract.}  In this paper, we study the unicorn's Landsberg problem from an intrinsic point of view. Precisely, we investigate a coordinate-free proof of Numata's theorem on Landsberg spaces of scalar curvature. In other words, following the pullback approach to Finsler geometry, we prove that all  Landsberg spaces of dimension $n\geq 3$ of non-zero scalar curvature are Riemannian spaces of constant curvature.
 



\vspace{30truept}\centerline{\Large\bf{Introduction}}\vspace{12pt}
\par

Let $ {(M,L)}$ be an n-dimensional smooth Finsler manifold. The manifold $ {(M,L)}$ is called a Berwald manifold if for any piecewise smooth curve $c(t)$ joining  two points $p,q\in M$, the Berwald parallel translation $P_c$ is linear isometry  between the tangent spaces $T_pM$  and $T_qM$.
 This is equivalent to that the geodesic spray of $L$ is quadratic.   Also, $ {(M,L)}$ is called a  Landsberg manifold  if the   parallel translation $P_c$ along $c$ preserves the induced Riemannian metrics on the slit tangent spaces $ T_pM \backslash \{0\}$ and $T_qM \backslash \{0\}$ is an isometry.
   This is equivalent to the property that the horizontal covariant derivative of the metric tensor of $F$ with respect to Berwald connection vanishes.
 
 It is clear that every Berwald space is  Landsberg. Whether there are Landsberg spaces which are not  Berwaldian  is a long-standing question in Finsler geometry, which is still open.  Despite the   efforts done  by many  geometers, it is not known   a regular non-Berwaldian Landsberg space.

 In \cite{Asanov}, G. S. Asanov  obtained  examples, arising from Finslerian General Relativity, of non-Berwaldian Landsberg spaces, of dimension at least $3$.  In Asanov's examples the Finsler function is not defined for all values of the fiber coordinates $y^i$ so it is a non-regular Finsler function. 
    In \cite{Shen_example}, Z. Shen  studied a class of $(\alpha,\beta)$ metrics of Landsberg type generalizing Asanov's example; he found  non-regular non-Berwaldian Landsberg spaces.         The elusiveness of regular non-Berwaldian Landsberg spaces leads Bao \cite{Bao} to describe them as the unicorns of Finsler geometry. In some special cases, a Landsberg manifold reduces to Berwald manifold. For example, S. Numata in \cite{Numata} has proved that all Landsberg metrics of $n\ge 3$ and of non-vanishing scalar curvature are Remannian metric of non-zero constant curvature.

 \bigskip

All work that are mentioned above are local study. On the other hand, there are very few papers studying the unicorn problem     intrinsically.  In the present paper, we treat the unicorn's Landsberg problem intrinsically.  Following the pullback approach to Finsler geometry, we study intrinsically Landsberg Finsler spaces of non-vanishing scalar curvature and providing an intrinsic proof of Numata's theorem. We prove a useful property on $C$-reducible Finsler spaces (cf. Proposition \ref{prop.3a}). Also, we show that   a Landsberg manifold of non zero scaler curvature  is $C$-reducible (cf. Proposition \ref{prop.3}). Then, we prove that a Berwald manifold of non zero scaler curvature   and  $n\geq3$ is a Riemannian manifold of constant curvature (cf. Theorem \ref{thm.3}). Finally, we conclude that a Landsberg manifold of non zero scaler curvature   and  $n\geq3$ is a Riemannian manifold of constant curvature (cf. Theorem \ref{thm.1}).

\section{Notation and Preliminaries}

Here, we present  some of the fundamental basics of the pullback approach to   Finsler geometry that are required for this study. For more detail about this approach, we refer, for example,  to \cite{ r93, amr3, r94, r96}. 

Let $M$ be an $n$-imensional smooth manifold, consider the  tangent bundle  $\pi: T M\longrightarrow M$ and its differential $d\pi: TT M\longrightarrow TM$. The vertical bundle $V(TM)$ of $TM$ is just $ \ker (d\pi)$. Let us denote the pullback bundle of the tangent bundle by $\pi^{-1}(T M)$.
 Also,  $\mathfrak{F}(TM)$ denotes the algebra of $C^\infty$ functions on $TM$ and $\cp$ the $\mathfrak{F}(TM)$-module of differentiable sections of the pullback bundle  $\pi^{-1}(T M)$. The elements of $\mathfrak{X}(\pi (M))$ will be called $\pi$-vector fields and denoted by barred letters $\overline{X}$. \\
 \vspace*{-0.4cm}
\par
Recall the short exact sequence of vector  bundle morphisms \cite{r21}
$$0\longrightarrow
 \pi^{-1}(TM)\stackrel{\gamma}\longrightarrow T(\T M)\stackrel{\rho}\longrightarrow
\pi^{-1}(TM)\longrightarrow 0 ,\vspace{-0.1cm}$$ where $\T M$ is the slit tangent bundle, $\gamma$ is the natural injection and $\rho := (\pi_{TM}, \pi)$.

The  tangent structure $J$ of $TM$ or the vertical endomorphism is the endomorphism $J: T\T M \longrightarrow T \T M$ defined by $J=\gamma \circ \rho$. The Liouville vector field $\mathcal{C}$ is  given by $ \, \mathcal{C}:=\gamma\, \overline{\eta}, $ where  $\overline{\eta}(u)=(u,u)$ for all $u\in \T M$.
\par
For a linear connection  $D$ on $\pi^{-1}(TM)$,     the associated connection map  $ K$ is defined by \vspace{-0.1cm} $K:T \T M\longrightarrow
\pi^{-1}(TM):X\longmapsto D_X \overline{\eta} $, and the horizontal space $H_u (\T M)$ to $M$ at $u$ is $H_u (\T M):= \{ X \in T_u
(\T M) : K(X)=0 \}$.  The connection $D$ is said to be regular if
$$ T_u (\T M)=V_u (\T M)\oplus H_u (\T M) \,\,\,  \forall \, u\in \T M.$$

For a regular connection $D$ on $M$, the vector bundle
   maps $  \gamma,\, \rho |_{H(\T M)}$ and $K |_{V(\T M)}$
 are isomorphisms. In this case, the map  $\beta:=(\rho |_{H(\T M)})^{-1}$
 is called the horizontal map of $D$.
\par
\begin{defn} Let $D$ be a regular connection on $\pi^{-1}(TM)$ with the
horizontal map   $\beta$ and   the corresponding  classical torsion (resp. curvature) tensor  field $\textbf{T}$ (resp. $\textbf{K}$).
  Then, we have
\begin{enumerate}
  \item For a $\pi$-tensor field  $A$  of type $(0,p)$,   the $h$- and $v$-covariant derivatives
$\stackrel{h}D$ and $\stackrel{v}D$:
\begin{eqnarray*}\label{covariat 1}
  (\stackrel{h}D A)(\o X, \overline{X}_{1},...,\overline{X}_{p})&:=&(D_{\beta \o
X}A)(\overline{X}_{1},...,\overline{X}_{p}).\\
  (\stackrel{v}D A)(\o X, \overline{X}_{1},...,\overline{X}_{p})&:=&(D_{\gamma \o
X}A)(\overline{X}_{1},...,\overline{X}_{p}).
\end{eqnarray*}
  \item The (h)h-, (h)hv- and (h)v-torsion tensors of $D$: \vspace{-0.2cm}
$$Q (\overline{X},\overline{Y}):=\textbf{T}(\beta \overline{X},\beta \overline{Y}),
\, \,\,\, T(\overline{X},\overline{Y}):=\textbf{T}(\gamma
\overline{X},\beta \overline{Y}) ,
\, \,\,\, V(\overline{X},\overline{Y}):=\textbf{T}(\gamma
\overline{X},\gamma \overline{Y}) ,\vspace{-0.2cm}$$
  \item The horizontal, mixed and vertical curvature tensors
of $D$:
$$R(\overline{X},\overline{Y})\overline{Z}:=\textbf{K}(\beta
\overline{X},\beta \overline{Y})\overline{Z},\quad
 {P}(\overline{X},\overline{Y})\overline{Z}:=\textbf{K}(\beta
\overline{X},\gamma \overline{Y})\overline{Z}, $$
$$
 {S}(\overline{X},\overline{Y})\overline{Z}:=\textbf{K}(\gamma
\overline{X},\gamma \overline{Y})\overline{Z}, $$

\item  The (v)h-, (v)hv- and (v)v-torsion tensors of $D$:
$$\widehat{ {R}}(\overline{X},\overline{Y}):={ {R}}(\overline{X},\overline{Y})\overline{\eta},\quad
\widehat{ {P}}(\overline{X},\overline{Y}):={
{P}}(\overline{X},\overline{Y})\overline{\eta},\quad \widehat{
{S}}(\overline{X},\overline{Y}):={
{S}}(\overline{X},\overline{Y})\overline{\eta}.$$
\end{enumerate}
\end{defn}


Throughout,  we assume that  $(M,L)$ is a Finsler manifold of dimension $n$. We have  the following geometric objects:
\begin{eqnarray*}
g &:& \text{the Finser metric defined by $L$},\\
\ell &:& \text{the normalized supporting element defined by $\ell :=L^{-1}i_{\overline{\eta}}\:g$},\\
\hbar &:& \text{the angular metric tensor defined by $\hbar := g-\ell \otimes \ell$},\\
\phi &:& \text{the vector $\pi$-form associated with $\hbar$ defined by $i_{\phi(\overline{X})}\,g:=i_{\overline{X}}\:\hbar$}\\
D^{\circ}&:&\text{the Berwald  connection associated with $(M,L)$},\\
\stackrel{h}{D^{\circ}} (\stackrel{v}{D^{\circ}})&:& \text{the horizontal (vertical) covariant derivative associated with $D^{\circ}$},\\
R^\circ,\, P^\circ, \, \widehat{R^{\circ}} &:& \text{the $h$-curvature, $hv$-curvature,  $(v)h$-torsion tensors of Berwald connection},\\
\mathbf{P}^\circ &:&\text{the Berwald $hv$-curvature of type $(0,4)$ defined by  }\\
&&\mathbf{P}^{\circ}(\overline{X},\overline{Y},\overline{Z},\overline{W}):=g(P^{\circ}(\overline{X},\overline{Y})\overline{Z},\overline{W}),\\
H:=i_{\overline{\eta}}\,\widehat{R^{\circ}} &:& \text{the deviation tensor of Berwald connection},\\
\nabla &:&\text{the Cartan connection associated with $(M,L)$},\\
\stackrel{h}{\nabla} (\stackrel{v}{\nabla})&:& \text{the horizontal (vertical) covariant derivative associated with $\nabla$},\\
R,\,  P, \, \widehat{R} &:& \text{the $h$-curvature, $hv$-curvature,  $(v)h$-torsion tensors of Cartan connection},\\
T  &:& \text{the $(h)hv$-torsion of Cartan connection},\\
C &:& \text{the contracted torsion form defined by $C(\overline{X}):= Tr\{\overline{Y} \longmapsto T(\overline{X},\overline{Y})\}$},\\
\mathbf{T}&:& \text{the  Cartan tensor defined by $\mathbf{T}(\overline{X},\overline{Y},\overline{Z}):=g(T(\overline{X},\overline{Y}),\overline{Z})$},\\
\widehat{P} &:&  \text{$(v)hv$-torsion tensor of Cartan connection}.
\end{eqnarray*}
The following result provides the relation between the Berwald connection $D^{\circ}$ and the Cartan connection $\nabla$.
\begin{prop}\label{prop.1} \emph{\cite{r92}} Let $(M,L)$ be a Finsler manifold and  $g$ be the Finsler metric induced by $L$. The Cartan connection $\nabla$ and the Berwald connection $D^{\circ}$  are related by:
\begin{description}
   \item[\em{\textbf{(a)}}] $ {{D}}^{\circ}_{\gamma \overline{X}}\overline{Y}=\nabla _{\gamma
  \overline{X}}\overline{Y}-T(\overline{X},\overline{Y})=\rho[\gamma
\overline{X}, \beta \overline{Y}]$.

  \item[\em{\textbf{(b)}}] $ {{D}}^{\circ}_{\beta \overline{X}}\overline{Y}=\nabla _{\beta
  \overline{X}}\overline{Y}+\widehat{P}(\overline{X},\overline{Y})=K[\beta
\overline{X}, \gamma \overline{Y}],$
\end{description}
where $K$ and  $\beta$ are the connection map and the horizontal map associated with Cartan  connection $\nabla$, respectively.
\end{prop}

\begin{defn}\label{sca.}{\em{\cite{r86}}} A Finsler manifold $(M,L)$ with $n\geq 3$ is
said to be of scalar curvature  $r$ if the deviation tensor $H$ satisfies the property 
$$H(\overline{X})=r L^{2} \phi(\overline{X}), $$
where $r$ is a scalar function on $\T M$, positively homogeneous of degree zero in $y$ ($h^{+}(0)$)\footnote{$\omega$ is $h^{+}(k)$ in $y$ iff $D^{\circ}_{\gamma \overline{\eta}}\,\omega=k \: \omega$.}.
In particular, if the scalar curvature $r$ is constant, then $(M,L)$ is called a Finsler manifold of constant curvature.
\end{defn}

\begin{defn}{\em{\cite{r86}}} \label{def.1} For a Finsler manifold $(M,L)$ is said to be {\em:}
\begin{description}
 \item[(a)]  Berwald   if  the Berwald hv-curvature $P^{\circ}=0 \, \Leftrightarrow \, \nabla_{\beta \overline{X}}\,T =0.$

 \item[(b)]  Landsberg   if the Cartan hv-curvature $P=0 \, \Leftrightarrow \, \nabla_{\beta \overline{\eta}}\,T =0 = \widehat{P}$.
 \end{description}
 \end{defn}

\section{C-reducible Finsler manifolds}

Let's start with the definition of $C$-reducible Finsler  manifolds.

\begin{defn} \label{def.2} {\em{\cite{r86}}}
A Finsler manifold $(M,L)$ is called $C$-reducible, if the Cartan tensor $\textbf{T}$ has the form
  \begin{equation*}
      \textbf{T}(\overline{X},\overline{Y},\overline{Z})=  \frac{1}{n+1} \{\hbar(\overline{X}
  ,\overline{Y})C(\overline{Z})+\hbar(\overline{Y}
  ,\overline{Z})C(\overline{X})+\hbar(\overline{Z},\overline{X})
  C(\overline{Y})\}\,
   \end{equation*}
where $C$ is the contracted torsion form.
\end{defn}

The following three lemmas are useful for  subsequent use.
\begin{lem}\label{lem.1}
For a Finsler manifold $(M,L)$, we have:
\begin{description}
 \item[{\textbf{(a)}}] $\mathbf{T}$,  $\stackrel{v} {\nabla}\mathbf{T}$  and $\hbar$ are totally symmetric,

\item[{\textbf{(b)}}]
$\stackrel{v}{\nabla}L=\stackrel{v}{D^{\circ}}L=\ell, \quad
\stackrel{v}{\nabla}\ell=\stackrel{v}{D^{\circ}}\ell=L^{-1}\hbar.$

\item[{\textbf{(c)}}]
$\stackrel{v}{D^{\circ}}\phi=-L^{-2} \hbar\otimes \overline{\eta}-L^{-1}\phi
\otimes \ell.$

 \item[{\textbf{(d)}}]$ ({D^{\circ}}_{\gamma \overline{X}} \, \hbar)(\overline{Y},\overline{Z})=2\textbf{T}(\overline{ X}, \overline{Y}, \overline{Z}) - L^{-1} \hbar(\overline{ X}, \overline{Y}) \ell(\overline{Z})- L^{-1} \hbar(\overline{ X}, \overline{Z}) \ell(\overline{Y}).$

      \item[{\textbf{(e)}}]   $ ({\nabla}_{\gamma \overline{X}} \, \hbar)(\overline{Y},\overline{Z})= - L^{-1} \hbar(\overline{X}, \overline{Y}) \ell(\overline{Z})- L^{-1}  \hbar(\overline{X}, \overline{Z}) \ell(\overline{Y})$.
      \end{description}
\end{lem}
\begin{proof} The proof is clear and we omit it.
\end{proof}

\bigskip

For a Finsler manifold $(M,L)$ of a non zero scalar curvature $r$, we define:
\begin{eqnarray}
  A(\overline{X}, \overline{Y}) &:=& L \ell(\overline{X}) D^{\circ}_{\gamma \overline{Y}}\:r +\frac{2}{3} L \ell{(\overline{Y})} D^{\circ}_{\gamma \overline{X}}\:r   +r\, \ell{(\overline{X})}\ell{(\overline{Y})}+ \frac{1}{3} L^2 D^{\circ}_{\gamma \overline{Y}}D^{\circ}_{\gamma \overline{X}}\:r \label{eq.1} \\
  B(\overline{X}) &:=& r L \ell{(\overline{X})}+\frac{1}{3} L^2 D^{\circ}_{\gamma \overline{X}}\:r  \label{eq.2} .
\end{eqnarray}

\begin{lem}\label{lem.2}
The tensor fields $A$ and $B$, defined above,  have the following properties
\begin{description}

\item[{\textbf{(a)}}]
$ A(\overline{\eta},\overline{X} )= r L \ell{(\overline{X})}+\frac{2}{3} L^2 D^{\circ}_{\gamma \overline{X}}\:r$

\item[{\textbf{(b)}}]
$A(\overline{X}, \overline{\eta})=B(\overline{X})$

 \item[{\textbf{(c)}}]$A(\overline{\eta}, \overline{\eta})=B(\overline{\eta})=r \, L^2$.

 \item[{\textbf{(e)}}]$ (D^{\circ}_{\gamma \overline{Y}} B) (\overline{X})=A(\overline{X},\overline{Y})+r \, \hbar(\overline{X},\overline{Y})$.

\end{description}
\end{lem}
\begin{proof} The proof follows from Lemma \ref{lem.1} taking into account the facts that $\ell(\overline{\eta})=L$ and $r$ is positively homogenous of degree zero in $y$ .
\end{proof}

\begin{lem} \label{lem.3} The $h$-curvature tensor ${R^{\circ}}$ of Berwald connection, for a Finsler manifold $(M,L)$ of non zero scalar curvature $r$, has the form\footnote{$\mathfrak{A}_{\overline{X},\overline{Y}} \{A({\overline{X},\overline{Y}})\}=A({\overline{X},\overline{Y}})-
A({\overline{Y},\overline{X}})$.}
$$
  {R^{\circ}}(\overline{X},\overline{Y})\overline{Z}=\mathfrak{A}_{\overline{X},\overline{Y}}
  \{[r \, \hbar(\overline{X},\overline{Z})+A(\overline{X},\overline{Z})]\phi(\overline{Y})
     -B(\overline{X})[L^{-2}\hbar(\overline{Y},\overline{Z})\, \overline{\eta}+L^{-1} \ell(\overline{Y})\phi(Z)]\}.
$$
\end{lem}
\begin{proof}
Let $(M,L)$ be a Finsler
manifold of non zero scaler curvature $r$. Then, by Definition \ref{sca.},  \cite[Theorem 4.6]{r96} and  Lemma \ref{lem.1},   we obtain
\begin{eqnarray}
   \widehat{ R^{\circ}}({\overline{X},\overline{Y}})&=&\frac{1}{3}\mathfrak{A}_{\overline{X},\overline{Y}}\set{(\stackrel{v}{D^{\circ}}H)(\overline{X},\overline{Y})} \nonumber \\
   &=& \mathfrak{A}_{\overline{X},\overline{Y}}\set{B(\overline{X}) \phi(\overline{Y})} \label{R hat}.
\end{eqnarray}
where $B$ is the tensor field given by  (\ref{eq.2}).
\par
On the other hand, again by  \cite[Theorem 4.6]{r96},  we have
\begin{equation*}
  R^{\circ}({\overline{X},\overline{Y}})\overline{Z}=(\stackrel{v}{D^{\circ}}\widehat{R^{\circ}})(\overline{Z},\overline{X},\overline{Y}).
\end{equation*}
From which, together with  (\ref{eq.1}) and Lemmas
\ref{lem.1} and \ref{lem.2}, after some computation the result follows.
\end{proof}

\begin{prop} \label{prop.3a}
For a $C$-reducible Finsler space there exists a scalar $\alpha(x,y)$ such
that
\begin{equation}\label{eq.8a}
L\,({\nabla}_{\gamma \overline{X}} C)(\overline{W})+\ell(\overline{X})C(\overline{W})+\ell(\overline{W})C(\overline{X})
=\alpha(x,y) \, \hbar(\overline{X}, \overline{W}).
\end{equation}
\end{prop}
\begin{proof} From Lemma \ref{lem.1}(a), we have
\begin{equation}\label{1a}
 ({\nabla}_{\gamma \overline{X}} \, \mathbf{T})(\overline{Y},\overline{Z},\overline{W})=({\nabla}_{\gamma \overline{Y}} \, \mathbf{T})(\overline{X},\overline{Z},\overline{W})
\end{equation}
Contracting $\overline{Z}$ with $\overline{W}$, the above relation reduces to
\begin{equation}\label{2a}
 ({\nabla}_{\gamma \overline{X}} \, C)(\overline{Y})=({\nabla}_{\gamma \overline{Y}} \, C)(\overline{X})
\end{equation}
Again from  (\ref{1a}), taking into account the $C$-reducibility property, we obtain
\begin{eqnarray*}
   ({\nabla}_{\gamma \overline{X}} \, \hbar)(\overline{Y},\overline{Z})\, C(\overline{W})+\hbar(\overline{Y},\overline{Z})\,
   ({\nabla}_{\gamma \overline{X}} \, C)(\overline{W}) &&  \\
 +({\nabla}_{\gamma \overline{X}} \, \hbar)(\overline{Z},\overline{W})\, C(\overline{Y})+\hbar(\overline{Z},\overline{W})\,
   ({\nabla}_{\gamma \overline{X}} \, C)(\overline{Y}) &&  \\
   +({\nabla}_{\gamma \overline{X}} \, \hbar)(\overline{W},\overline{Y})\, C(\overline{Z})+\hbar(\overline{W},\overline{Y})\,
   ({\nabla}_{\gamma \overline{X}} \, C)(\overline{Z}) && \\
    -({\nabla}_{\gamma \overline{Y}} \, \hbar)(\overline{X},\overline{Z})\, C(\overline{W})-\hbar(\overline{X},\overline{Z})\,
   ({\nabla}_{\gamma \overline{Y}} \, C)(\overline{W}) &&  \\
 -({\nabla}_{\gamma \overline{Y}} \, \hbar)(\overline{Z},\overline{W})\, C(\overline{X})-\hbar(\overline{Z},\overline{W})\,
   ({\nabla}_{\gamma \overline{Y}} \, C)(\overline{X}) &&  \\
   -({\nabla}_{\gamma \overline{Y}} \, \hbar)(\overline{W},\overline{X})\, C(\overline{Z})-\hbar(\overline{W},\overline{X})\,
   ({\nabla}_{\gamma \overline{Y}} \, C)(\overline{Z}) &=&0
\end{eqnarray*}
 Applying  Lemma \ref{lem.1}(e) and  (\ref{2a}),  we obtain that
 \begin{equation}\label{3a}
   \hbar(\overline{Y},\overline{Z})\mathbb{A}(\overline{X},\overline{W})
   +\hbar(\overline{Y},\overline{W})\mathbb{A}(\overline{X},\overline{Z})
   -\hbar(\overline{X},\overline{Z})\mathbb{A}(\overline{Y},\overline{W})
   -\hbar(\overline{X},\overline{W})\mathbb{A}(\overline{Y},\overline{Z})=0,
 \end{equation}
 where $ \mathbb{A}$ is a $\pi$-tensor field of type $(0,2)$ defined by
 \begin{equation}\label{4a}
  \mathbb{A}(\overline{X},\overline{W}):=({\nabla}_{\gamma \overline{X}} \, C)(\overline{W})+L^{-1}
  \{\ell(\overline{X})\,C(\overline{W})+\ell(\overline{W})\,C(\overline{X})\}.
 \end{equation}
 Contracting $\overline{Y}$ with $\overline{W}$ into  (\ref{3a}), we get
 \begin{equation*}\label{5a}
\mathbb{A}(\overline{X},\overline{Z})+(n-1)\mathbb{A}(\overline{X},\overline{Z})-f(x,y)
 \,\hbar(\overline{X},\overline{Z})-\mathbb{A}(\overline{X},\overline{Z})=0,
 \end{equation*}
where  $f(x,y)$ is the contracting $\overline{Y}$ with $\overline{W}$ for the $\pi$-tensor field $\mathbb{A}(\overline{Y},\overline{W})$.
From which together with the expression of $\mathbb{A}$ (\ref{4a}), the result follows where $\alpha(x,y)
:=\frac{f(x,y) L}{(n-1)}$.
\end{proof}

\section{Landsberg $C$-reducible manifolds}

It is obvious that every Berwald manifold is Landsberg, but the converse is not true.
However, the following two results generalize the results of Matsumoto \cite{r30}: \vspace{-0.1cm}
\begin{thm}\label{prop.2}  A $C$-reducible  Landsberg
 manifold $(M,L)$, with dimension $n\geq3$, is  Berwaldian  or Riemaniann.
\end{thm}
 \begin{proof} Let $(M,L)$ be a $C$-reducible  Landsberg
 manifold. Hence, from Definition \ref{def.1}, we conclude that the Cartan $hv$-curvature  $P$ and  $\widehat{P}$ vanish identically. Consequently, using  \cite[Theorem (3.5)(c)]{r96} taking into account the fact that $\nabla g=0$, we have
 \begin{equation}\label{11a}
 ( \nabla_{\beta \overline{Z}}\mathbf{T})(\overline{X}, \overline{Y},\overline{W})
  =( \nabla_{\beta \overline{W}}\mathbf{T})(\overline{X},\overline{Y},\overline{Z}).
 \end{equation}
Contracting $\overline{X}$ with $\overline{Y}$ implies
\begin{equation}\label{12a}
 ( \nabla_{\beta \overline{Z}}\,C)(\overline{W})
  =( \nabla_{\beta \overline{W}}\,C)(\overline{Z}).
 \end{equation}
Hence, for $C$-reducible manifold together with (\ref{11a}) and the property $\stackrel{h}\nabla \hbar=0$, we obtain
\begin{eqnarray*}
  &&\hbar(\overline{Y}
  ,\overline{Z})( \nabla_{\beta \overline{W}}\,C)(\overline{X})+\hbar(\overline{Z},\overline{X})
  ( \nabla_{\beta \overline{W}}\,C)(\overline{Y})\\
  & =& \hbar(\overline{Y}  ,\overline{W})( \nabla_{\beta \overline{Z}}\,C)(\overline{X})+\hbar(\overline{W},\overline{X})
  ( \nabla_{\beta \overline{Z}}\,C)(\overline{Y}).
\end{eqnarray*}
   Contracting $\overline{X}$ with $\overline{Z}$ for both sides and using (\ref{12a}), one can show that
   \begin{eqnarray*}
    & & \sigma(x,y)\,\hbar(\overline{Y}  ,\overline{W})+  ( \nabla_{\beta \overline{W}}\,C)(\overline{Y})-L^{-1}\,\ell(\overline{W})
    ( \nabla_{\beta \overline{\eta}}\,C)(\overline{Y})\\
  &=& ( \nabla_{\beta \overline{W}}\,C)(\overline{Y})-L^{-1}\,\ell(\overline{Y})
    ( \nabla_{\beta \overline{W}}\,C)(\overline{\eta})+ (n-1)( \nabla_{\beta \overline{W}}\,C)(\overline{Y}),
\end{eqnarray*}
where $\sigma(x,y)$ is the contracting $\overline{X}$ with $\overline{Z}$ for the term $( \nabla_{\beta \overline{Z}}\,C)(\overline{X})$.  From which taking into account the fact that $( \nabla_{\beta \overline{W}}\,C)(\overline{\eta})$ vanishes identically, we get
\begin{eqnarray}\label{q1}
   (\nabla_{\beta \overline{W}}\,C)(\overline{Y})&=&\mu(x,y)\,\hbar(\overline{Y}  ,\overline{W})
   \Longleftrightarrow 
   \nabla_{\beta \overline{W}}\,\overline{C} = \mu(x,y)\,\phi(\overline{W}),  
\end{eqnarray}
with $\mu(x,y):=\frac{\sigma(x,y)}{(n-1)}$ and $C(\overline{X})=:g(\overline{C},\overline{X})$.

In general, for Cartan connection  \cite[Theorem 3.4]{r96}, we have
\begin{eqnarray*}
 (\nabla_{\beta\overline{Z}}S)(\overline{X},\overline{Y},\overline{W}
)&=&(\nabla_{\gamma\overline{X}}P)(\overline{Z},\overline{Y},\overline{W})-
(\nabla_{\gamma \overline{Y}}P)(\overline{Z},\overline{X},
\overline{W})-S(\widehat{P}(\overline{Z},\overline{Y}),\overline{X})\overline{W} \\
&+&S(\widehat{P}(\overline{Z},\overline{X}),\overline{Y})\overline{W}
-P(T(\overline{Y},\overline{Z}),\overline{X})\overline{W}+ P(T(\overline{X},\overline{Z}),\overline{Y})\overline{W}.
\end{eqnarray*}
In case of Landsberg manifold, due to Definition \ref{def.1}(b), we conclude that
\begin{equation*}
  (\nabla_{\beta\overline{Z}}S)(\overline{X},\overline{Y},\overline{W})=0
\end{equation*}
 From which taking into account  \cite[Theorem 3.4]{r96}, we get
 \begin{eqnarray*}
  &&g((\nabla_{\beta \overline{N}}T)(\overline{X},\overline{W}) , T(\overline{Y}, \overline{Z}))+
  g(T(\overline{X},\overline{W}) , (\nabla_{\beta \overline{N}}T)(\overline{Y}, \overline{Z}))\\
  && -g((\nabla_{\beta \overline{N}}T)(\overline{Y}, \overline{W}) ,T( \overline{X}, \overline{Z}))-g(T(\overline{Y}, \overline{W}) ,(\nabla_{\beta \overline{N}}T)( \overline{X}, \overline{Z}))=0.
\end{eqnarray*}
 Hence,  for a $C$-reducible manifold taking into account   (\ref{q1}), one can show that
 \begin{eqnarray*}
   &&(n+1)^{-1}\mu(x,y)\{
    \hbar(\overline{X},\overline{W})\mathbf{T}(\overline{Y},\overline{Z},\overline{N})
   +\hbar(\overline{X},\overline{N})\mathbf{T}(\overline{Y},\overline{Z},\overline{W})
   +\hbar(\overline{W},\overline{N})\mathbf{T}(\overline{Y},\overline{Z},\overline{X}) \\
    &+&\hbar(\overline{Y},\overline{Z})\mathbf{T}(\overline{X},\overline{W},\overline{N})
   +\hbar(\overline{Y},\overline{N})\mathbf{T}(\overline{X},\overline{W},\overline{Z})
   +\hbar(\overline{Z},\overline{N})\mathbf{T}(\overline{X},\overline{W},\overline{Y})\\
    &-&\hbar(\overline{Y},\overline{W})\mathbf{T}(\overline{X},\overline{Z},\overline{N})
   -\hbar(\overline{Y},\overline{N})\mathbf{T}(\overline{X},\overline{Z},\overline{W})
   -\hbar(\overline{W},\overline{N})\mathbf{T}(\overline{X},\overline{Z},\overline{Y}) \\
    &-&\hbar(\overline{X},\overline{Z})\mathbf{T}(\overline{Y},\overline{W},\overline{N})
   -\hbar(\overline{X},\overline{N})\mathbf{T}(\overline{Y},\overline{W},\overline{Z})
   -\hbar(\overline{Z},\overline{N})\mathbf{T}(\overline{Y},\overline{W},\overline{X})\}=0.
 \end{eqnarray*}
Contracting $\overline{X}$ with $\overline{W}$, the above equation reduces to
\begin{equation*}
\mu(x,y)\,\{(n-3)\mathbf{T}(\overline{Y},\overline{Z},\overline{N})+\hbar(\overline{Y},\overline{Z})\,C(\overline{N})\}=0.
\end{equation*}
Again contracting $\overline{Y}$ with $\overline{Z}$, we obtain
\begin{equation}\label{nn}
(n-2)\mu(x,y)\,C(\overline{N})=0. 
\end{equation}
Therefore, provided that $n\geq3$, we have two cases:\\
{\bf Case 1:} if  $\mu(x,y)\neq 0$, then the contracted torsion $C$ vanishes. Hence, the Cartan torsion $T=0$ by reducibility property. Consequently, $(M,L)$ is Riemannian. \\
{\bf Case 2:} if  $\mu(x,y)= 0$, then by Equation (\ref{q1}) the horizontal covariant derivatives for the contracted torsion $C$ vanishes identically. Hence, the horizontal covariant derivative  $(\nabla_{\beta \overline{W}}\, \mathbf{T})=0$ by reducibility property, which means that $(M,L)$ is Berwald. This completes the proof.
 \end{proof}
  
\begin{prop} \label{prop.3} If $(M,L)$ is a Landsberg manifold of non zero scaler curvature $r$, then it is a $C$-reducible manifold.
\end{prop}
\begin{proof}  By \cite[ Theorem 4.6]{r96}, we have:
$$(D^{\circ}_{\gamma \overline{X}}R^{\circ})(\overline{Y},
    \overline{Z},  \overline{W})=(D^{\circ}_{\beta \overline{Z}}P^{\circ})(\overline{Y},
    \overline{X},  \overline{W})-(D^{\circ}_{\beta \overline{Y}}P^{\circ})(\overline{Z},
    \overline{X},  \overline{W}).$$
Setting $\overline{Z}=\overline{\eta}$ noting the facts that $i_{\overline{\eta}}P^{\circ}=0$ and $K \circ \beta=0$, we obtain   $$(D^{\circ}_{\beta \overline{\eta}}P^{\circ})(\overline{Y},
    \overline{X},  \overline{W})=(D^{\circ}_{\gamma \overline{X}}R^{\circ})(\overline{Y},
    \overline{\eta},  \overline{W}).$$

 Since $(M,L)$ is a Finsler manifold of non zero scaler curvature $r$, then from the above relation and Lemma \ref{lem.3}, we get
 \begin{eqnarray*}
   (D^{\circ}_{\beta \overline{\eta}}P^{\circ})(\overline{Y},
    \overline{X},  \overline{W}) &=& -2 L^{-3}\ell(\overline{X}) \hbar(\overline{Y}, \overline{W}) B(\overline{\eta}) \, \overline{\eta}
    -L^{-2}(D^{\circ}_{\gamma \overline{X}}\hbar)(\overline{\eta}, \overline{W})B(\overline{Y}) \, \overline{\eta}
    \\
    && +L^{-2}(D^{\circ}_{\gamma \overline{X}}\hbar)(\overline{Y}, \overline{W})B(\overline{\eta}) \, \overline{\eta}
    +L^{-2} \hbar(\overline{Y}, \overline{W})(D^{\circ}_{\gamma \overline{X}}B)(\overline{\eta}) \, \overline{\eta}\\
    &&+L^{-2} \hbar(\overline{Y}, \overline{W})B(\overline{\eta}) \, D^{\circ}_{\gamma \overline{X}}\overline{\eta}
    - r \,(D^{\circ}_{\gamma \overline{X}}\hbar)(\overline{\eta}, \overline{W}) \phi(\overline{Y})    \\
    &&  + r \, \hbar(\overline{Y}, \overline{W})(D^{\circ}_{\gamma \overline{X}}\phi)(\overline{\eta})
    + L^{-2} \ell(\overline{X}) \ell(\overline{\eta}) B(\overline{Y})\phi(\overline{W})\\
    && - L^{-2} \ell(\overline{X}) \ell(\overline{Y}) B(\overline{\eta})\phi(\overline{W})
    - L^{-1}  \ell(\overline{\eta}) B(\overline{Y})(D^{\circ}_{\gamma \overline{X}}\phi)(\overline{W})\\
    &&+ L^{-1}  \ell(\overline{Y}) B(\overline{\eta})(D^{\circ}_{\gamma \overline{X}}\phi)(\overline{W})
    - L^{-1}  (D^{\circ}_{\gamma \overline{X}}\ell)(\overline{\eta}) B(\overline{Y})\phi(\overline{W})\\
    && + L^{-1}  (D^{\circ}_{\gamma \overline{X}}\ell)(\overline{Y}) B(\overline{\eta})\phi(\overline{W})
    -L^{-1}  \ell(\overline{\eta}) (D^{\circ}_{\gamma \overline{X}}B)(\overline{Y})\phi(\overline{W})\\
    && + L^{-1}  \ell(\overline{Y}) (D^{\circ}_{\gamma \overline{X}}B)(\overline{\eta})\phi(\overline{W})
    + A(\overline{Y}, \overline{W}) (D^{\circ}_{\gamma \overline{X}}\phi)(\overline{\eta}) \\
    && -  A(\overline{\eta}, \overline{W}) (D^{\circ}_{\gamma \overline{X}}\phi)(\overline{Y})
     + (D^{\circ}_{\gamma \overline{X}}A)(\overline{Y}, \overline{W})\phi(\overline{\eta})
     - (D^{\circ}_{\gamma \overline{X}}A)(\overline{\eta}, \overline{W}) \phi(\overline{Y}) .
    \end{eqnarray*}
 Thus, using the facts that $(D^{\circ}_{\beta \overline{\eta}}\textbf{P}^{\circ})(\overline{Y},
    \overline{X},  \overline{W}, \overline{Z})=g((D^{\circ}_{\beta \overline{\eta}}P^{\circ})(\overline{Y},
    \overline{X},  \overline{W}), \overline{Z})$, $i_{\overline{\eta}}\phi=0=i_{\overline{\eta}}\hbar$, together with Lemmas \ref{lem.1}
    , \ref{lem.2} and \ref{lem.3},  after long calculations, we have
       \begin{eqnarray}
      (D^{\circ}_{\beta \overline{\eta}}\textbf{P}^{\circ})(\overline{Y},
    \overline{X},  \overline{W}, \overline{Z}) &=& \frac{2}{3} L \ell(\overline{Z})
    [\hbar(\overline{X},\overline{W}) D^{\circ}_{\gamma \overline{Y}}\, r
     + \hbar(\overline{Y},\overline{W}) D^{\circ}_{\gamma \overline{X}}\, r  \nonumber     \\
    && +\hbar(\overline{X},\overline{Y}) D^{\circ}_{\gamma \overline{W}}\, r
     +3r \,\textbf{T}(\overline{X}, \overline{Y},\overline{W}) ]
      -\frac{1}{3}[\hbar(\overline{Y},\overline{Z}) \mathbf{M}(\overline{X},\overline{W}) \nonumber \\
      &&  +\hbar(\overline{X},\overline{Z}) \mathbf{M}(\overline{Y},\overline{W})
     +\hbar(\overline{W},\overline{Z}) \mathbf{M}(\overline{X},\overline{Y})] \label{eq.1a} ,
                \end{eqnarray}
 where
 \begin{equation}\label{eq.M}
   \mathbf{M}(\overline{X},\overline{Y}):=L \, \ell(\overline{X})D^{\circ}_{\gamma \overline{Y}}\, r
   +  L\,\ell(\overline{Y})D^{\circ}_{\gamma \overline{X}}\, r + L^2 D^{\circ}_{\gamma \overline{X}}D^{\circ}_{\gamma \overline{Y}}\, r.
 \end{equation}
 Putting $\overline{Z}=\overline{\eta}$,   we get
    \begin{eqnarray}
      (D^{\circ}_{\beta \overline{\eta}}\textbf{P}^{\circ})(\overline{Y},
    \overline{X},  \overline{W}, \overline{\eta}) &=& \frac{2}{3} L^2 [\hbar(\overline{X},\overline{W}) D^{\circ}_{\gamma \overline{Y}}\, r
     + \hbar(\overline{Y},\overline{W}) D^{\circ}_{\gamma \overline{X}}\, r  \nonumber     \\
    && +\hbar(\overline{X},\overline{Y}) D^{\circ}_{\gamma \overline{W}}\, r
     +3r \,\textbf{T}(\overline{X}, \overline{Y},\overline{W}) ] \label{eq.4},
                \end{eqnarray}
 On the other hand, by \cite{r92}, we have:
 \begin{eqnarray*}
   {{P}}^{\circ}(\overline{X},\overline{Y})\overline{Z}&=& P(\overline{X},\overline{Y})\overline{Z}+
  (\nabla_{\gamma \overline{Y}}\widehat{P})(\overline{X},\overline{Z}) +
   \widehat{P}(T(\overline{Y},\overline{X}),\overline{Z})
   +\widehat{P}(\overline{X},T(\overline{Y},\overline{Z}))\\
&&+(\nabla_{\beta\overline{X}}T)(\overline{Y},\overline{Z})
-T(\overline{Y}, \widehat{P}(\overline{X},\overline{Z}))-T(
\widehat{P}(\overline{X}, \overline{Y}), \overline{Z}).
 \end{eqnarray*}
Hence, using Definition \ref{def.1} taking into account the fact that $\nabla_{\beta \overline{X}}g=0$, $(M,L)$ is landsberg if and only if 
$\textbf{P}^{\circ}(\overline{Y},\overline{X},  \overline{W}, \overline{\eta})$ vanishes identically.  Consequently, for a landsberg manifold,   (\ref{eq.4}) reduces to
    \begin{eqnarray}
      \textbf{T}(\overline{X}, \overline{Y},\overline{W}) &=& \frac{-1}{3 \, r}  [\hbar(\overline{X},\overline{W}) D^{\circ}_{\gamma \overline{Y}}\, r
     + \hbar(\overline{Y},\overline{W}) D^{\circ}_{\gamma \overline{X}}\, r
      +\hbar(\overline{X},\overline{Y}) D^{\circ}_{\gamma \overline{W}}\, r] \label{eq.5},
                \end{eqnarray}
 provided that $r\neq 0$. Contracting $\overline{Y}$ with $\overline{W}$, we obtain
 \begin{equation}\label{eq.6}
  D^{\circ}_{\gamma \overline{X}}\, r=\frac{- 3\, r}{(n+1)}  C(\overline{X}).
 \end{equation}
 From which together with (\ref{eq.5}), we conclude that $(M,L)$ is $C$-reducible.
 \end{proof}

\begin{thm} \label{thm.3} If $(M,L)$ is a Berwald manifold of non zero scaler curvature $r$ with $n\geq3$, then it is a Riemannian manifold of constant curvature.
\end{thm}
\begin{proof} Assume that $(M,L)$ is a Berwald manifold of non zero scaler curvature $r$ with $n\geq3$, then it is  $C$-reducible.
  Also, the Berwald hv-curvature $P^{\circ}$  vanishes identically  and   (\ref{eq.5}) holds good. Therefore   (\ref{eq.1a}) becomes
$$\hbar(\overline{Y},\overline{Z}) \mathbf{M}(\overline{X},\overline{W})
       +\hbar(\overline{X},\overline{Z}) \mathbf{M}(\overline{Y},\overline{W})
     +\hbar(\overline{W},\overline{Z}) \mathbf{M}(\overline{X},\overline{Y})=0$$
Contracting $\overline{Y}$ with $\overline{Z}$, we obtain
$$(n+1) \mathbf{M}(\overline{X},\overline{W})=0$$
Hence, from   (\ref{eq.M}), we conclude that
$$\ell(\overline{X})D^{\circ}_{\gamma \overline{W}}\, r
   +\ell(\overline{W})D^{\circ}_{\gamma \overline{X}}\, r + L D^{\circ}_{\gamma \overline{X}}D^{\circ}_{\gamma \overline{W}}\, r.=0$$
From which together with (\ref{eq.6}), we have
\begin{equation}\label{eq.7}
\ell(\overline{X})C(\overline{W})+\ell(\overline{W})C(\overline{X})+L [(D^{\circ}_{\gamma \overline{X}} C)(\overline{W})
-\frac{3}{(n+1)}C(\overline{X})C(\overline{W})]=0.
\end{equation}

On the other hand, for a $C$-reducible Finsler space and using Proposition \ref{prop.1}, one can show that
$$(\nabla_{\gamma \overline{X}} C)(\overline{W})=(D^{\circ}_{\gamma \overline{X}} C)(\overline{W})
-\frac{1}{(n+1)}\{C^{2} \, \hbar(\overline{X},\overline{W})+2C(\overline{X})C(\overline{W})\},$$
where $C^2:=C(\overline{C}); C(\overline{X})=:g(\overline{C},\overline{X})$.
Consequently, using Proposition \ref{prop.3a} we conclude that for a $C$-reducible Finsler manifold there exists a scalar $\psi(x,y)$ such
that
\begin{equation}\label{eq.8}
\ell(\overline{X})C(\overline{W})+\ell(\overline{W})C(\overline{X})+L [(D^{\circ}_{\gamma \overline{X}} C)(\overline{W})
-\frac{2}{(n+1)}C(\overline{X})C(\overline{W})]=\psi(x,y) \, \hbar(\overline{X}, \overline{W}),
\end{equation}
where $\psi(x,y):=\frac{L \, C^2}{(n+1)}+\alpha(x,y)$.  Now, from Eqs. (\ref{eq.7}) and (\ref{eq.8}), we get
$$\frac{L}{(n+1)}\, C(\overline{X})C(\overline{W})=\psi(x,y) \, \hbar(\overline{X}, \overline{W}),$$
As the trace of L.H.S. ($\hbar(\overline{X}, \overline{W})$) equals $n-1\geq2$ and the trace of R.H.S. ($C(\overline{X})C(\overline{W})$) equals $1$, then $\psi(x,y)$ vanishes identically. Hence, the torsion form $C=0$, and by $C$-reducibility the Cartan torsion $\textbf{T}$ vanishes. Therefore $(M,L)$ is a Riemannian Manifold. Also, from the fact that the torsion form $C$ vanishes together  (\ref{eq.6}), we conclude that
   \begin{equation}\label{eq.ver.cov.}
     D^{\circ}_{\gamma \overline{X}}\, r=0,
   \end{equation}
  which means that the scalar curvature $r$ vertically parallel. To prove that the scalar curvature $r$ is constant, we need to show that the scalar curvature $r$ is horizonally parallel as follows: \\
    By (\ref{R hat}), together with (\ref{eq.ver.cov.}), we obtain
 \begin{equation}\label{eq.9}
    \widehat{R^{\circ}}(\overline{X},\overline{Y})=r \, L \set{\ell(\overline{X})\overline{Y}-\ell(\overline{Y})\overline{X}}.
 \end{equation}
By \cite{r96}, we have   
\begin{equation*}\label{eq.10}
   \mathfrak{S}_{\overline{X},\overline{Y},\overline{Z}}\,
\{(D^{\circ}_{\beta \overline{X}}R^{\circ})(\overline{Y},
\overline{Z},\overline{W})+P^{\circ}(\widehat{R^{\circ}}(\overline{X},\overline{Y}),
\overline{Z})\overline{W}\}=0,
\end{equation*}
where $\mathfrak{S}_{\overline{X},\overline{Y},\overline{Z}}$ is the cyclic sum over $\overline{X},\overline{Y},\overline{Z}$.
 Hence, by \cite{r96},  the $(v)hv$-torsion $\widehat{P^{\circ}}$ vanishes, it follows that
\begin{equation}\label{eq.11}
   \mathfrak{S}_{\overline{X},\overline{Y},\overline{Z}}\,
(D^{\circ}_{\beta \overline{X}}\widehat{R^{\circ}})(\overline{Y},
\overline{Z})=0.
\end{equation}
In view of  (\ref{eq.9}) and (\ref{eq.11}),  Definition \ref{sca.} and the fact that $D^{\circ}_{\beta
\overline{X}}\ell=0$, we get
\begin{eqnarray*}
  && L(D^{\circ}_{\beta
\overline{X}}\,r)(\ell(\overline{Y})\overline{Z}-\ell(\overline{Z})\overline{Y})
+  L(D^{\circ}_{\beta
\overline{Y}}\,r)(\ell(\overline{Z})\overline{X}-\ell(\overline{X})\overline{Z}) \\
   &&+L(D^{\circ}_{\beta
\overline{Z}}\,r)(\ell(\overline{X})\overline{Y}-\ell(\overline{Y})\overline{X})=0.
\end{eqnarray*}
Setting $\overline{Z}=\overline{\eta}$ into the above equation,
noting that $\ell(\overline{\eta})=L$, we obtain
\begin{eqnarray*}
  && L({D^{\circ}}_{\beta
\overline{X}} \,r)(\ell(\overline{Y})\overline{\eta}-L\overline{Y})
+  L(D^{\circ}_{\beta
\overline{Y}} \,r)(L\overline{X}-\ell(\overline{X})\overline{\eta}) \\
   &&+L(D^{\circ}_{\beta
\overline{\eta}} \,r)(\ell(\overline{X})\overline{Y}-\ell(\overline{Y})\overline{X})=0.
\end{eqnarray*}
Taking the trace of both sides with respect to $\overline{Y}$, it
follows that
\begin{equation}\label{eq.12}
{D^{\circ}}_{\beta \overline{X}} \,r=L^{-1}({D^{\circ}}_{\beta
\overline{\eta}} \,r)\ell(\overline{X}).
\end{equation}
Applying the  vertical covariant derivative with respect to $\overline{Y}$
on both sides of (\ref{eq.12}), yields
\begin{equation*}
  \ell(\overline{Y})
  D^{\circ}_{\beta\overline{X}} \,r+L(\stackrel{v}{D^{\circ}}\stackrel{h}{D^{\circ}}
  \, r)(\overline{X},\overline{Y})=
  L^{-1}\hbar(\overline{X},\overline{Y})(D^{\circ}_{\beta\overline{\eta}} \,r)+\ell(\overline{X})
  (\stackrel{v}{D^{\circ}}\stackrel{h}{D^{\circ}} \, r)(\overline{\eta},\overline{Y}).
\end{equation*}
From (\ref{eq.ver.cov.}), noting that $(\stackrel{v}{D^{\circ}}\stackrel{h}{D^{\circ}}
  r )(\overline{X},\overline{Y})=(\stackrel{h}{D^{\circ}}\stackrel{v}{D^{\circ}}
   r)(\overline{Y},\overline{X})$, the above relation reduces to (provided that $n\geq3$)
\begin{equation*}
  \ell(\overline{Y})
  D^{\circ}_{\beta\overline{X}}\, r=
  L^{-1}\hbar(\overline{X},\overline{Y})(D^{\circ}_{\beta\overline{\eta}} \, r).
\end{equation*}
Setting $\overline{Y}=\overline{\eta}$ into the above equation,
noting that $\ell(\overline{\eta})=L$ and
$\hbar(.,\overline{\eta})=0$, it follows that
$D^{\circ}_{\beta\overline{X}} \, r=0$. Consequently,
\begin{equation}\label{eq.hor.cov.}
 D^{\circ}_{\beta \overline{X}}\, r=0,
\end{equation}
which means that the scalar curvature $r$ is horizonally parallel. Now, from (\ref{eq.ver.cov.}) and (\ref{eq.hor.cov.}), we conclude that $r$ is
constant. Consequently, the proof is complete.
\end{proof}

Finally, we provide a global proof of the Numata's theorem \cite{Numata} for Finsler manifold of a non-vanishing scalar curvature by incorporating previous results.
\begin{thm}\label{thm.1}  If $(M,L)$ is a Landsberg manifold of a non-vanishing scalar curvature $r$ with $n\geq 3$, then it is a Riemannian manifold of constant curvature.
\end{thm}

\begin{proof}  The proof follows from Theorem \ref{prop.2}, Proposition \ref{prop.3} and Theorem \ref{thm.3}.
\end{proof}


\providecommand{\bysame}{\leavevmode\hbox
to3em{\hrulefill}\thinspace}
\providecommand{\MR}{\relax\ifhmode\unskip\space\fi MR }
\providecommand{\MRhref}[2]{%
  \href{http://www.ams.org/mathscinet-getitem?mr=#1}{#2}
} \providecommand{\href}[2]{#2}

\end{document}